\documentclass{amsart}

\usepackage{amssymb}
\usepackage{amsthm}

\usepackage[usenames]{color}
\definecolor{ANDREW}{RGB}{255,127,0}

\usepackage[foot]{amsaddr}

\usepackage{tikz}
\usetikzlibrary{matrix,arrows,decorations.pathmorphing}

\usepackage{hyperref}

\theoremstyle{plain}
\newtheorem{proposition}{Proposition}[section]
\newtheorem{theorem}[proposition]{Theorem}
\newtheorem{lemma}[proposition]{Lemma}
\newtheorem{corollary}[proposition]{Corollary}
\theoremstyle{definition}

\newtheorem{definition}[proposition]{Definition}

\theoremstyle{remark}
\newtheorem{remark}[proposition]{Remark}

\DeclareMathOperator{\Aut}{Aut}

\DeclareMathOperator{\End}{End}
\DeclareMathOperator{\GL}{GL}

\DeclareMathOperator{\Proj}{Proj}

\DeclareMathOperator{\SL}{SL}

\DeclareMathOperator{\Prop}{Prop}
\DeclareMathOperator{\Ext}{Ext}
\DeclareMathOperator{\Imag}{Im}
\DeclareMathOperator{\PGL}{PGL}
\DeclareMathOperator{\Id}{Id}
\DeclareMathOperator{\Lin}{Lin}
\DeclareMathOperator{\kob}{kob}

\DeclareMathOperator{\Cc}{\mathcal{C}}

\DeclareMathOperator{\Uc}{\mathcal{U}}

\DeclareMathOperator{\Ab}{\mathbb{A}}
\DeclareMathOperator{\Cb}{\mathbb{C}}

\DeclareMathOperator{\Hb}{\mathbb{H}}

\DeclareMathOperator{\Nb}{\mathbb{N}}

\DeclareMathOperator{\Pb}{\mathbb{P}}
\DeclareMathOperator{\Rb}{\mathbb{R}}
\DeclareMathOperator{\Tb}{\mathbb{T}}
\DeclareMathOperator{\Zb}{\mathbb{Z}}

\newcommand{\abs}[1]{\left|#1\right|}

\newcommand{\norm}[1]{\left\|#1\right\|}
\newcommand{\wt}[1]{\widetilde{#1}}
\newcommand{\wh}[1]{\widehat{#1}}
\newcommand{\ip}[1]{\left\langle #1 \right\rangle}

\begin{document}

\title[Maps between real projective manifolds]{The structure of projective maps between real projective manifolds}
\author{Andrew M. Zimmer}\address{Department of Mathematics, University of Chicago, Chicago, IL 60637.}
\email{aazimmer@uchicago.edu}
\date{\today}
\keywords{Real projective manifolds, normal families, Hilbert metric, Kobayashi metric}
\subjclass[2010]{}

\begin{abstract} In this paper we study the set of projective maps between compact proper convex real projective manifolds. We show that this set contains only finitely many distinct homotopy classes and each homotopy class has the structure of a real projective manifold. When the target manifold is strictly convex, our results imply that each non-trivial homotopy class contains at most one projective map. These results are motivated by the  theory of holomorphic maps between compact complex manifolds. \end{abstract}

\date{\today}

\maketitle

\section{Introduction}

A real projective structure on a manifold $M$ is an open cover $M = \cup_{\alpha} U_\alpha$ along with coordinate charts $\varphi_\alpha : U_{\alpha} \rightarrow \Pb(\Rb^{d+1})$  such that each transition function $\varphi_{\alpha} \circ \varphi_{\beta}^{-1}$ coincides with the restriction of an element in $\PGL_{d+1}(\Rb)$. A \emph{real projective manifold} is a manifold equipped with a real projective structure. Precise definitions are given in Section~\ref{sec:real_proj_geom}.

An important class of real projective manifolds are the so-called \emph{convex real projective manifolds}. These are the real projective manifolds that can be identified as a quotient $M= \Gamma \backslash \Omega$ where $\Gamma \leq \PGL_{d+1}(\Rb)$ is a discrete group acting properly discontinuously and freely on a convex open set $\Omega \subset \Pb(\Rb^{d+1})$. Such a manifold is called \emph{proper} if $\Omega$ is a proper convex set and \emph{strictly convex} if $\Omega$ is a strictly convex set. More background can be found in the survey papers by Benoist~\cite{B2008}, Goldman~\cite{G2009}, Marquis~\cite{M2013}, and Quint~\cite{Q2010}.

Many compact manifolds have a convex real projective structure, for instance: every real hyperbolic manifold; the locally symmetric spaces associated to $\SL_d(\Rb)$, $\SL_d(\Cb)$, $\SL_d(\Hb)$, and $E_{6(-26)}$; many examples in low dimensions (see for instance \cite{B2006, V1971, VK1967}); and Kapovich~\cite{K2007} has shown that many of the Gromov-Thurston examples of manifolds with negative curvature have a strictly convex real projective structure. Moreover, some of these examples have a non-trivial moduli space of real projective structures. 

A \emph{projective map} $f:M_1 \rightarrow M_2$ between two projective manifolds is a map where $\varphi_{\alpha} \circ f \circ \phi_{\beta}^{-1}$ is the restriction of a projective map for any coordinate chart $\phi_{\beta}$ of $M_1$ and $\varphi_{\alpha}$ of $M_2$. Let $\Proj(M_1,M_2)$ denote the space of projective maps endowed with the compact-open topology and let $\Aut(M)$ denote the projective homeomorphisms $M \rightarrow M$. 

In this paper we study the set of projective maps between proper convex real projective manifolds. One of our main results is the following finiteness theorem:

\begin{theorem}\label{thm:main}
Suppose $M_1$ and $M_2$ are compact proper convex real projective manifolds. If $M_2$ is strictly convex then the set of non-constant maps in $\Proj(M_1,M_2)$ is finite. Moreover, each non-trivial homotopy class contains at most one real projective map. 
\end{theorem}

 Benoist~\cite{B2004} has shown that the fundamental group of a compact strictly convex real projective manifold is Gromov hyperbolic and for such structures there is a natural geodesic flow which is Anosov. So, in some sense, strictly convex real projective manifolds can be thought of as being negatively curved. Thus Theorem~\ref{thm:main} can be seen as a real projective analogue of the finiteness of isometries of a negatively curved compact manifold~\cite{B1946} or the uniqueness of harmonic maps in a homotopy class when the target manifold is compact and negatively curved~\cite{H1967}.
 
The two non-equivalent real projective structures on $\Tb^d$, the $d$-torus, show that both properness and strict convexity are necessary in Theorem~\ref{thm:main}.  First, we can identify $\Tb^d$ with $\Gamma_1 \backslash \Omega_1$ where
\begin{align*}
 \Omega_1 := \{ [1: x_1 : \dots : x_d ] : x_1, \dots, x_d \in \Rb\}
 \end{align*}
and
 \begin{align*}
\Gamma_1 : = \left\{ \begin{pmatrix} \Id_d & z \\ 0 & 1 \end{pmatrix} : z \in \Zb^d\right\}.
 \end{align*}
This structure is not proper and the real projective automorphism group of $M_1$ coincides with $\SL_d(\Zb)$, a non-compact group. We can also identify $\Tb^d$  with $\Gamma_2 \backslash \Omega_2$ where
  \begin{align*}
 \Omega_2 := \{ [1: x_1 : \dots : x_d ] : x_i > 0 \text{ for } 1 \leq i \leq d\}
 \end{align*}
and
 \begin{align*}
\Gamma_2 : = \left\{ \begin{pmatrix} e^{z_1} & & \\ & \ddots & \\ & & e^{z_{d+1}} \end{pmatrix} : z_1, \dots, z_{d+1} \in \Zb\right\}.
 \end{align*}
This structure is proper but not strictly convex. Here the real projective automorphism group coincides with $\operatorname{Sym}(d) \ltimes \Tb^d$ where $\operatorname{Sym}(d)$ is the symmetric group on $d$ symbols. In particular, the automorphism group is compact but not discrete. 
  
These two examples reflect the general theory: properness will always imply the set of projective maps is compact while strict convexity will imply discreteness. 

Every proper convex real projective manifold has a complete metric called the \emph{Hilbert metric} (defined in Section~\ref{sec:metric}) and Kobayashi~\cite{K1977} proved that every projective map is distance decreasing with respect to the Hilbert metrics. The distance decreasing property immediately implies that the set of projective maps between two compact proper convex real projective manifolds is pre-compact. A simple argument will actually show that it is compact.

 \begin{proposition}\label{prop:compact}
Suppose $M_1$ and $M_2$ are compact proper convex real projective manifolds. Then $\Proj(M_1,M_2)$ is compact. In particular, only finitely many homotopy classes can be represented by a  real projective map. 
\end{proposition}

As mentioned above, strict convexity is responsible for the set of projective maps being discrete. The idea is to lift two homotopic maps $f_1, f_2 : \Gamma_1 \backslash \Omega_1 \rightarrow \Gamma_2 \backslash \Omega_2$ to maps $\wt{f}_1, \wt{f}_2 : \Omega_1 \rightarrow \Omega_2$ and then study the induced maps of the boundaries. Using the strict convexity and the asymptotic geometry of the Hilbert metic we will show that the two maps agree on the boundary which will imply that they agree on the interior. 

In the case in which the target manifold is not strictly convex we have no guarantee of discreteness, but we can show that each connected component has a real projective structure. For a map $f_0:M_1 \rightarrow M_2$ between two real projective manifolds define the set 
\begin{align*}
[f_0]:=\left\{ f \in \Proj(M_1,M_2) : f \sim f_0\right\}.
\end{align*}

\begin{theorem}\label{thm:main_structure}
Given a non-constant projective map $f_0: M_1 \rightarrow M_2$ between two compact proper convex real projective manifolds the set $[f_0]$ is either $\{f_0\}$ or has the structure of a compact proper convex real projective manifold which is compatible with the compact-open topology. Moreover, 
with this structure for any fixed $m \in M_1$ the map 
\begin{align*}
f \in [f_0] \rightarrow f(m) \in M_2
\end{align*}
 is projective.
\end{theorem}

As a corollary to the proof of Theorem~\ref{thm:main} we will establish:

\begin{corollary}\label{cor:center}
Given a non-constant projective map $f: M_1 \rightarrow M_2$ between two compact proper convex real projective manifolds either $[f]=\{f\}$ or $f_*(\pi_1(M_1,m))$ has non-trivial centralizer in $\pi_1(M_2, f(m))$.
\end{corollary}

Given a linear map $T: \Rb^{d_1} \rightarrow \Rb^{d_2}$ one can always find a surjective linear map $S:\Rb^{d_1} \rightarrow \Rb^r$ and an injective linear map $I:\Rb^{r} \rightarrow \Rb^{d_2}$ so that $T = I \circ S$. As the next result shows an analogous result holds for projective maps and the surjective map can be chosen to depend only on the homotopy class. 

\begin{proposition}\label{prop:factor}
Suppose that $f_0: M_1 \rightarrow M_2$ is a non-constant projective map between two compact proper convex real projective manifolds. Then there exists a compact proper convex real projective manifold $N$ and a surjective projective map $p:M_1 \rightarrow N$ with the following property: if $f \in [f_0]$ then there exists a locally injective projective map $\overline{f}:N \rightarrow M_2$ with $f = \overline{f} \circ p$.
\end{proposition}

Given $f_0: M_1 \rightarrow M_2$ a non-constant projective map between two compact proper convex real projective manifolds there are two obvious ways for $[f_0]$ to be infinite. First the projective automorphism group of $M_1$ could be infinite and then $\{ f_0 \circ \phi_t : t \in \Rb\} \subset [f_0]$ for any one-parameter subgroup of $\Aut(M_1)$. Second $f_0$ could factor through a product $M_1 \rightarrow M_1 \times L \rightarrow M_2$. Then if the map $M_1 \times L \rightarrow M_2$ is projective on each fiber we again have infinitely many maps in $[f_0]$. The next theorem shows that these are essentially the only two possibilities:

\begin{theorem}\label{thm:prod_or_aut}
With the notation in Proposition~\ref{prop:factor}, $[f_0] \neq \{f_0\}$ if and only if either 
\begin{enumerate}
\item $\Aut(N)$ is infinite 
\item there exists a compact proper convex real projective manifold $L$ and a continuous locally injective map $F:N \times L \rightarrow M_2$ such that 
\begin{enumerate}
\item for any fixed $\ell \in L$ the map $n \in N \rightarrow F(n,\ell)$ is projective,
\item for any fixed $n \in N$ the map $\ell \in L \rightarrow F(n,\ell)$ is projective, and
\item there exists $\ell_0 \in L$ such that $f_0(m) = F(p(m),\ell_0)$ for all $m \in M_1$.
\end{enumerate}
\end{enumerate}
\end{theorem}

\begin{remark} \
\begin{enumerate}
\item In general a product of real projective manifolds does not have a real projective structure and so $F$ being projective on each fiber is the best one can expect. 
\item If $[f_0] \neq \{f_0\}$ then we know from Corollary~\ref{cor:center} that $(f_0)_*(\pi_1(M_1,m))$ has non-trivial centralizer $C$ in $\pi_1(M_2, f_0(m))$. The first case happens when 
\begin{align*}
C \cap (f_0)_*(\pi_1(M_1,m)) \neq 1
\end{align*}
while the second case happens when 
\begin{align*}
C \cap (f_0)_*(\pi_1(M_1,m)) = 1.
\end{align*}
\item As we will see, the manifold $L$ coincides with $[\overline{f}_0]$ and the map $F$ is simply $F(n,\ell) = \ell(n)$. 
\end{enumerate}
\end{remark}

\subsection{A characterization of proper convex real projective manifolds}

The starting point in this study is Proposition~\ref{prop:compact} which uses the distance decreasing property of the Hilbert metric. The Hilbert metric is usually only defined for convex real projective manifolds (see Section~\ref{sec:metric}), but in fact every real projective manifold has a natural pseudo-metric. In particular, Kobayashi~\cite{K1977} introduced the following Finsler pseudo-metric on a real projective manifold
\begin{align*}
h^{\kob}_M(p; v) = \inf \{ \abs{\xi} : f:(-1,1) \rightarrow M \text{ is projective}, f(0)=p, df_0(\xi) = v\}
\end{align*}
and the integrated pseudo-distance
\begin{align*}
H^{\kob}_M(p,q) = \inf \left\{ \int_0^1 h^{\kob}_M(\sigma(t); \sigma^\prime(t)) dt : \sigma : [0,1] \rightarrow M \text{ is } C^{\infty}, \sigma(0)=p, \sigma(1)=q\right\}.
\end{align*}
By definition these metrics are distance decreasing with respect to projective maps. Moreover, Kobayashi proved that for convex real projective manifolds this metric coincides with the Hilbert metric. One could then hope to generalize the results of this paper to the real projective manifolds for which this pseudo-metric is non-degenerate. Unfortunately Kobayashi proved the following:

\begin{theorem}\cite{K1977}\label{thm:characterization}
Suppose $M$ is a compact real projective manifold. Then the following are equivalent: 
\begin{enumerate}
\item If $\Ab \subset \Pb(\Rb^2)$ is an affine chart then every real projective map $\Ab \rightarrow M$ is constant, 
\item $H_M^{\kob}$ is non-degenerate, 
\item $M$ is a compact proper convex real projective manifold. 
\end{enumerate}
\end{theorem}

\subsection{Complex Manifolds} Some of the results above are motivated by the theory of holomorphic maps between complex manifolds. In particular, any complex manifold has a possibly degenerate metric called the \emph{Kobayashi metric}. The Kobayashi metric has the remarkable property that any holomorphic map $f:M_1 \rightarrow M_2$ is distance decreasing. In some cases, for instance complex projective space, the Kobayashi metric is trivial but when it is an actual metric one can use this distance decreasing property to establish compactness of holomorphic maps.  Moreover for compact manifolds, there is a simple (to state) characterization due to Brody of when the metric is non-degenerate.

\begin{theorem}\cite[Theorem 4.1]{B1978} Suppose $M$ is a compact complex manifold. Then the Kobayashi metric is non-degenerate if and only if every holomorphic map $\Cb \rightarrow M$ is constant. 
\end{theorem}

So Theorem~\ref{thm:main} can be seen as an analogue of the following theorem in complex analysis:

\begin{theorem}\label{thm:complex_finite}\cite[pg. 70]{K1970}
Suppose $M$ is a compact complex manifold. If every holomorphic map $\Cb \rightarrow M$ is constant then the group of bi-hilomorphisms of $M$ is finite. 
\end{theorem}

We should remark that the standard proof Theorem~\ref{thm:complex_finite} is different from the arguments in our paper. Bochner and Montgomery proved that the bi-holomorphism group of an compact complex manifold is a complex Lie group~\cite{BM1947}. In the case in which the Kobayashi metric is non-degenerate, it acts by isometries and thus will be  compact Lie group. Then the group is either finite or contains a complex one-parameter group $\{ \exp( z X) : z \in \Cb\}$. In the latter case, for a generic $m \in M$ the map $z \rightarrow \exp(zX) \cdot m$ is a non-constant holomorphic map of $\Cb \rightarrow M$ which is impossible when the Kobayashi metric is non-degenerate.

One might hope to simply repeat the proof of Theorem~\ref{thm:complex_finite} in the real projective setting. However, this will have little hope of succeeding because the projective automorphism group of $\Tb^d$ with its proper projective structure need not be finite. 


\subsection*{Acknowledgments} 

I would like to thank Benson Farb for many helpful conversations, in particular  explaining why the set of bi-holomorphisms between two compact hyperbolic surfaces is finite and asking whether a similar result holds for real projective manifolds. This material is based upon work supported by the National Science Foundation under Grant Number NSF 1045119 and Grant Number NSF 1400919.

\section{Notation} 

Given some object $o$ we will let $[o]$ be the projective equivalence class of $o$, for instance: if $v \in \Rb^{d+1} \setminus \{0\}$ let $[v]$ denote the image of $v$ in $\Pb(\Rb^{d+1})$, if $\phi \in \GL_{d+1}(\Rb)$ let $[\phi]$ denote the image of $\phi$ in $\PGL_{d+1}(\Rb)$, and if $T \in \Lin(\Rb^{d_1+1}, \Rb^{d_2+1}) \setminus\{0\}$ let $[T]$ denote the image of $T$ in $\Pb(\Lin(\Rb^{d_1+1}, \Rb^{d_2+1}))$. 

\section{Convex real projective manifolds}\label{sec:real_proj_geom}

 A \emph{real projective atlas} on a second countable Hausdorff space $M$ is a pair $(\Uc, \Phi)$ where $\Uc=\left\{ U_\alpha \right\}$ is an open covering of $M$ and
 $\Phi = \left\{ \phi_{\alpha}:U_\alpha \rightarrow \Pb(\Rb^{d+1}) \right\}$ is a collection of homeomorphisms onto open sets of $\Pb(\Rb^{d+1})$ such that $\phi_{\alpha} \circ \phi_{\beta}^{-1}$ is the restriction of some element of $\PGL_{d+1}(\Rb)$ to $\phi_{\beta}(U_{\alpha} \cap U_{\beta})$. Given an atlas, a pair $(U_{\alpha}, \phi_{\alpha})$ is called a \emph{chart} on $M$.
 
 \begin{definition}
 A \emph{real projective structure} on a second countable Hausdorff space $M$ is a maximal projective atlas $(\Uc,\Phi)$. A topological space $M$ equipped with a projective structure is called a \emph{projective manifold. }
 \end{definition}
 
If $f:M \rightarrow N$ is a local diffeomorphism and $N$ has a real projective structure, there is a unique real projective structure on $M$ making the map $f: M \rightarrow N$ a \emph{projective map} (see below).  So if  $M$ is a projective manifold, we can lift the projective structure to the universal cover $\wt{M}$ of $M$. Then fixing a base point $\wt{x}_0 \in \wt{M}$ over some $x_0 \in M$ we can use our charts to define a homomorphism $\operatorname{hol}: \pi_1(M,x_0) \rightarrow \PGL_{d+1}(\Rb)$ called the \emph{holonomy map} and a local diffeomorphism $\operatorname{dev}: \wt{M} \rightarrow \Pb(\Rb^{d+1})$ called the \emph{developing map}. In general these maps are only defined up to conjugation in $\PGL_{d+1}(\Rb)$ and the developing map will be equivariant with respect to the holonomy map. All these definitions come from the general theory of geometric structures on manifolds and details can be found in~\cite[Section 2]{G1988}.

A set $\Omega \subset \Pb(\Rb^{d+1})$ is \emph{convex} if $L \cap \Omega$ is connected and $L \cap \Omega \neq L$ for any projective line $L \subset \Pb(\Rb^{d+1})$. A convex set $\Omega$ is \emph{proper} if $\overline{L \cap \Omega} \neq L$ for any projective line $L \subset \Pb(\Rb^{d+1})$. Finally a convex set $\Omega$ is called \emph{strictly convex} if the intersection of $\partial \Omega$ with any projective line is at most two points. It is possible to show that $\Omega$ is a proper convex open set if and only if there exists an affine chart containing $\Omega$ as a bounded convex open set. 
 
 \begin{definition}
 A projective manifold $M$ is called \emph{convex} if the developing map $\operatorname{dev} : \wt{M} \rightarrow \Pb(\Rb^{d+1})$ is a diffeomorphism onto a convex open set $\Omega$ in $\Pb(\Rb^{d+1})$. A convex projective manifold is called \emph{proper} if $\Omega$ is a proper convex set. A convex projective manifold is called \emph{strictly convex} if $\Omega$ is a strictly convex set.
 \end{definition} 

By definition every convex real projective manifold $M$ can be identified with a quotient $\Gamma \backslash \Omega$ where 
\begin{align*}
\Omega:=\operatorname{dev}(\wt{M}) \subset \Pb(\Rb^{d+1})
\end{align*}
is a convex open set and 
\begin{align*}
\Gamma:=\operatorname{hol}(\pi_1(M,x_0)) \leq \PGL_{d+1}(\Rb)
\end{align*}
is a discrete group which acts properly discontinuously and freely on $\Omega$.
 
\subsection{Projective maps}\label{subsec:proj_maps} 

Given two real projective manifolds $M_1$ and $M_2$, we say a map $f: M_1 \rightarrow M_2$ is a \emph{projective map} if for each chart $\phi_{\alpha} : U_\alpha \rightarrow \Pb(\Rb^{d_1+1})$ on $M_1$ and chart $\phi_{\beta} : V_\beta \rightarrow \Pb(\Rb^{d_2+1})$ on $M_2$, the composition $\phi_{\beta} \circ f \circ \phi_{\alpha}^{-1}$ coincides on $\phi_{\alpha}(U_\alpha \cap f^{-1}(V_\beta))$ with some map of the form $x \rightarrow T(x)$ where $T \in \Pb(\Lin(\Rb^{d_1+1}, \Rb^{d_2+1}))$. 

As the next result shows when $M_1 = \Gamma_1 \backslash \Omega_1$ and $M_2 = \Gamma_2 \backslash \Omega_2$ are convex real projective manifold any projective map $f :M_1 \rightarrow M_2$ can be lifted to a projective map $T : \Omega_1 \rightarrow \Omega_2$. 
 
\begin{proposition}\label{prop:proj_map}
Suppose $M_1=\Gamma_1 \backslash \Omega_1$ and $M_2=\Gamma_2 \backslash \Omega_2$ are convex projective manifolds and $f:M_1 \rightarrow M_2$ is a non-constant projective map. Let $\wt{f}:\Omega_1 \rightarrow \Omega_2$ be a lift of $f$. Then 
\begin{enumerate}
\item $\wt{f}:\Omega_1 \rightarrow \Omega_2$ is a projective map, 
\item there exists a unique $T \in \Pb(\Lin(\Rb^{d_1+1}, \Rb^{d_2+1}))$ 
such that 
\begin{align*}
\wt{f}(p)=T(p)  \text{ for all } p \in \Omega_1,
\end{align*}
\item there exists a non-trivial homomorphism $\rho:\Gamma_1 \rightarrow \Gamma_2$ such that 
\begin{align*}
T( \gamma p) = \rho(\gamma)T(p)
\end{align*}
for all $p \in \Omega_1$ and $\gamma \in \Gamma_1$, 
\item $[\ker T]$ is $\Gamma_1$-invariant. 
\end{enumerate}
 
\end{proposition}

\begin{remark}
Suppose $S_1, S_2 \in \Lin(\Rb^{d_1+1}, \Rb^{d_2+1})$ are two rank one linear maps with the same image but $\ker S_1 \neq \ker S_2$. If $\Omega_1$ is a convex open set and $[\ker S_1 \cup \ker S_2]$ does not intersect $\Omega_1$ then $[S_1]$ and $[S_2]$ induce the same projective map on $\Omega_1$ even though $[S_1] \neq [S_2]$ as elements of $\Pb(\Lin(\Rb^{d_1+1}, \Rb^{d_2+1}))$. This example leads to the assumption that $f$ is non-constant in Proposition~\ref{prop:proj_map}.
\end{remark}

We begin the proof of the proposition with some lemmas. 

\begin{lemma}
\label{lem:proj_univ}
Suppose $\Omega_1 \subset \Pb(\Rb^{d_1+1})$ and $\Omega_2 \subset \Pb(\Rb^{d_2+1})$ are proper open convex sets. If $f: \Omega_1 \rightarrow \Omega_2$ is a projective map then there exists $T \in \Pb(\Lin(\Rb^{d_1+1}, \Rb^{d_2+1}))$ such that $f(p)=T(p)$ for all $p \in \Omega_1$.
\end{lemma}

\begin{proof} This follows from the fact that $(\Omega_1, \Id)$ and $(\Omega_2, \Id)$ are charts on $\Omega_1$ and $\Omega_2$ respectively.
\end{proof}

\begin{lemma}
\label{lem:cont}
Suppose $T  \in \Pb(\Lin(\Rb^{d_1+1}, \Rb^{d_2+1}))$ is a map with rank at least two. Then 
\begin{align*}
[\ker T] = \{ x \in \Pb(\Rb^{d_1+1}) : \lim_{y \rightarrow x} T(y) \text{ does not exist} \}.
\end{align*}
\end{lemma}

\begin{proof}
Since the map $x \rightarrow T(x)$ is well defined and analytic on $\Pb(\Rb^{d_1+1}) \setminus [\ker T]$ we have the $\supseteq$ direction.

Now suppose $v \in \ker T$. Since $T$ has rank at least two there exists $e_1, e_2 \in \Rb^{d_1+1} \setminus \ker T$ such that $T([e_1]) \neq T([e_2])$ then 
\begin{align*}
\lim_{n\rightarrow \infty} T\left( \left[v+\frac{1}{n}e_1\right]\right) = T([e_1])\neq T([e_2]) = \lim_{n\rightarrow \infty} T\left( \left[v+\frac{1}{n}e_2\right]\right).
\end{align*}

\end{proof}

\begin{lemma}
\label{lem:rank2}
Suppose $T_1,T_2  \in \Pb(\Lin(\Rb^{d_1+1}, \Rb^{d_2+1}))$ each have rank at least two. If there exists an open set $U \subset \Pb(\Rb^{d_1+1})$ such that $T_1(x) = T_2(x)$ for all $x \in U$ then  $T_1 = T_2$.
\end{lemma}

\begin{proof} Since $T_1$ and $T_2$ each have rank at least two, the set $V=\Pb(\Rb^{d_1+1}) \setminus [\ker T_1 \cup \ker T_2]$  is connected. Moreover $x \rightarrow T_1(x)$ and $x \rightarrow T_2(x)$ are analytic on $V$. Since they agree on an open set, they must agree on all of $V$. Then using Lemma~\ref{lem:cont} we see that $\ker T_1 = \ker T_2$ and then that $T_1(x) = T_2(x)$ for all $x \in \Pb(\Rb^{d_1+1}) \setminus [\ker T_i]$. This implies that $T_1 = T_2$.
\end{proof}

\begin{proof}[Proof of Proposition~\ref{prop:proj_map}]
For $i \in \{1,2\}$ let $\pi_i : \Omega_i \rightarrow M_i$ be the covering map. Notice that $\pi_1$ and $\pi_2$ are projective maps and local diffeomorphisms.  Since $\pi_2 \circ \wt{f} = f \circ \pi_1$ and the projectivity of a map is a local condition we see that $\wt{f}$ is projective as well. By Lemma~\ref{lem:proj_univ}  there exists $T \in \Pb(\Lin(\Rb^{d_1+1}, \Rb^{d_2+1}))$ 
such that 
\begin{align*}
T(p) = \wt{f}(p) \text{ for all } p \in \Omega_1.
\end{align*}
Since $\wt{f}$ is non-constant, $T$ has rank at least two. Thus, by Lemma~\ref{lem:rank2}, $T$ is unique.

Since $\wt{f}$ is a lift of $f$ there exists a homomorphism $\rho : \Gamma_1 \rightarrow \Gamma_2$ such that 
\begin{align*}
T( \gamma p) = \wt{f}(\gamma p ) = \rho(\gamma) \wt{f}(p)= \rho(\gamma)T(p)
\end{align*}
for all $p \in \Omega_1$ and $\gamma \in \Gamma_1$. We claim that $\rho$ is non-trivial. Fix some $p \in \Omega$. By Lemma~\ref{lem:orb_ext} 
\begin{align*}
\Ext(\Omega_1) \subset \overline{\Gamma_1 \cdot p}.
\end{align*}
Then if $\rho$ was trivial, we would have that $T(\Ext(\Omega_1)) = T(p)$ which would imply that $T(\Omega_1) = T(p)$. But this is impossible since $f$ is non-constant. 

Now assume  $\gamma \in \Gamma_1$. Then $(T \circ \gamma)(p) = (\rho(\gamma) \circ T)(p)$ for all $p \in \Omega_1$. Hence by Lemma~\ref{lem:rank2}, $(T \circ \gamma) = (\rho(\gamma) \circ T)$. Then 
\begin{align*}
\gamma^{-1} [\ker T] = [\ker (T \circ \gamma)] = [\ker ( \rho(\gamma) \circ T)] = [\ker T].
\end{align*}
Since $\gamma \in \Gamma_1$ was arbitrary this implies that $[\ker T]$ is $\Gamma_1$-invariant.
\end{proof}

\section{Convex cones, irreducibility, and lifting projective maps}\label{sec:cones}

At times it will be helpful to work in affine space instead of projective space. To this end, suppose $M = \Gamma \backslash \Omega$ is a compact proper convex real projective manifold. Since $\Omega$ is convex the preimage of $\Omega$ under the map $\Rb^{d+1} \setminus \{0\} \rightarrow \Pb(\Rb^{d+1})$ has two components $\Cc$ and $-\Cc$. Each component is a convex open cone which does not contain any affine lines. Let 
\begin{align*}
\overline{\Gamma} = \{ \gamma \in \SL_{d+1}^{\pm}(\Rb) : \gamma(\Cc)  =\Cc \text{ and } [\gamma] \in \Gamma\}
\end{align*}
where 
\begin{align*}
\SL_{d+1}^{\pm}(\Rb) = \{ \gamma \in \GL_{d+1}(\Rb) : \det (\gamma) = \pm 1\}.
\end{align*}
Notice that the map $\gamma \rightarrow [\gamma]$ induces an isomorphism $\overline{\Gamma} \rightarrow \Gamma$, since $\Cc \cap (-\Cc) = \emptyset$ and so $\gamma$ and $-\gamma$ cannot both preserve $\Cc$. Now let 
\begin{align*}
\Lambda := \ip{ \overline{\Gamma}, e \Id} \cong \Gamma \times \Zb
\end{align*}
then $\Lambda$ is discrete and acts co-compactly on $\Cc$.

\subsection{Reducibility}

\begin{definition} \
\begin{enumerate}
\item A convex cone $\Cc \subset \Rb^{d+1}$ is called \emph{proper} if $\Cc$ does not contain any affine lines. 
\item A proper convex cone $\Cc \subset \Rb^{d+1}$ is called \emph{reducible} if there exists a decomposition $\Rb^{d+1}=V_1 \oplus V_2$ and proper convex cones $\Cc_1 \subset V_1$ and $\Cc_2 \subset V_2$ such that
\begin{align*}
\Cc = \Cc_1 + \Cc_2 = \{ (c_1,c_2) : c_1 \in \Cc_1, c_2 \in \Cc_2\}.
\end{align*}
A proper convex cone is \emph{irreducible} if it is not reducible. 
\end{enumerate}
\end{definition}

Vey proved that when $M = \Gamma \backslash \Omega$ is a proper convex real projective manifold and $\Omega$ is irreducible then a discrete lift of $\Gamma$ acts irreducibly on $\Rb^{d+1}$. More precisely:

\begin{theorem}\label{thm:irred}\cite[Theorem 3, Theorem 5]{V1970}
Suppose $\Cc \subset \Rb^{d+1}$ is a proper convex open cone and $\Lambda \leq \GL_{d+1}(\Rb)$ is a discrete group acting co-compactly on $\Cc$. Then there exists a $\Lambda$-invariant decomposition $\Rb^{d+1}=V_1 \oplus V_2 \oplus \dots \oplus V_k$ such that 
\begin{enumerate}
\item $\Lambda$ acts irreducibly on each $V_i$
\item for each $1 \leq i \leq k$ there is a proper convex cone $\Cc_i \subset V_i$ such that 
\begin{align*}
\Cc = \Cc_1 + \dots + \Cc_k.
\end{align*} 
\end{enumerate}
\end{theorem}

\subsection{Lifting projective maps to linear maps}\label{subsec:linear} Now suppose that $M_1 = \Gamma_1 \backslash \Omega_1$ and $M_2 = \Gamma_2 \backslash \Omega_2$ are two compact proper convex real projective manifolds.  If $f:M_1 \rightarrow M_2$ is a non-constant projective map then let $T :\Omega_1 \rightarrow \Omega_2$ and $\rho: \Gamma_1 \rightarrow \Gamma_2$ be the maps from Proposition~\ref{prop:proj_map}. Let $\Cc_1, \Cc_2$ be cones above $\Omega_1, \Omega_2$ and $\Lambda_1, \Lambda_2$ be the groups constructed at the start of this section. Now we can lift $\rho:\Gamma_1 \rightarrow \Gamma_2$ to a homomorphism $\tau:\Lambda_1 \rightarrow \Lambda_2$ and the projective map $T : \Omega_1 \rightarrow \Omega_2$ to a linear map $S:\Cc_1 \rightarrow \Cc_2$ so that
\begin{align*}
S(\phi p) = \tau(\phi)S(p)
\end{align*} 
for all $\phi \in \Lambda_1$ and $p \in \Cc_1$. 

\section{The Hilbert metric}\label{sec:metric}

For distinct points $x,y \in \Pb(\Rb^{d+1})$ let $\overline{xy}$ be the projective line containing them. Suppose $\Omega \subset \Pb(\Rb^{d+1})$ is a proper convex open set. If $x,y \in \Omega$ let $a,b$ be the two points in $\overline{xy} \cap \partial\Omega$ ordered $a, x, y, b$ along $\overline{xy}$. Then define 
\begin{align*}
d_{\Omega}(x,y) = \log [a, x,y, b]
\end{align*}
 where 
 \begin{align*}
 [a,x,y,b] = \frac{\abs{x-b}\abs{y-a}}{\abs{x-a}\abs{y-b}}
 \end{align*}
 is the cross ratio. 
 
For an open set $\Omega \subset \Pb(\Rb^{d+1})$ define
 \begin{align*}
 \Aut(\Omega) = \{ \varphi \in \PGL_{d+1}(\Rb) : \varphi(\Omega) = \Omega \}.
 \end{align*}
 
 \begin{proposition}\label{prop:hilbert_basic}\textnormal{(see for instance~\cite[Corollary 3.4]{G2009})}
Suppose $\Omega \subset \Pb(\Rb^{d+1})$ is a proper convex open set. Then $d_{\Omega}$ is a complete $\Aut(\Omega)$-invariant metric on $\Omega$ which generates the standard topology on $\Omega$. 
\end{proposition}

As an immediate corollary we have:

\begin{corollary}
Suppose $\Omega \subset \Pb(\Rb^{d+1})$ is a proper convex open set.  Then $\Aut(\Omega) \leq \PGL_{d+1}(\Rb)$ is a closed Lie subgroup which acts properly on $\Omega$.
\end{corollary}

If $M=\Gamma \backslash \Omega$ is a proper convex real projective manifold and $\pi:\Omega \rightarrow M$ is the covering map then we can define 
\begin{align*}
d_{M}(p,q) = \inf \left\{ d_{\Omega}(\wt{p},\wt{q}) : \wt{p} \in \pi^{-1}(p) \text{ and } \wt{q} \in \pi^{-1}(q) \right\}.
\end{align*} 
Then $d_{M}$ will be a complete metric invariant under the projective automorphisms of $M$. Kobayashi proved the following distance decreasing property of the Hilbert metric:

 \begin{lemma}\cite{K1977}
 \label{lem:contract}
 Suppose $f: M_1 \rightarrow M_2$ is a projective map of two proper convex projective manifolds. Then 
 \begin{align*}
 d_{M_2}(f(p), f(q)) \leq d_{M_1}(p,q)
 \end{align*}
 for all $p,q \in M_1$.
\end{lemma}

A proof of Lemma~\ref{lem:contract} can also be found in~\cite[Proposition 3.3]{G2009}. 

\subsection{The asymptotic geometry of the Hilbert metric} There are deep connections between the shape of $\partial \Omega$ and the asymptotic geometry of the Hilbert metric on $\Omega$ (see for instance~\cite{B2003a, B2004,C2014, KN2002}). The next two observations are well known but since the proofs are short we will include them. 

\begin{definition}
A \emph{line segment} in $\Pb(\Rb^{d+1})$ is a connected subset of a projective line. 
\end{definition}

  \begin{lemma}
  \label{lem:bd_behavior}
  Suppose $\Omega \subset \Pb(\Rb^{d+1})$ is a proper convex open set and $p_n \in \Omega$ is a sequence with $p_n \rightarrow p \in \partial \Omega$ in the ambient topology of $\Pb(\Rb^{d+1})$. If $q_n \in \Omega$ is a sequence such that $q_n \rightarrow q$ and $\lim_{n \rightarrow \infty} d_{\Omega}(p_n,q_n) < +\infty$  then either 
  \begin{enumerate}
  \item $p=q$ or 
  \item $p$ and $q$ are in the interior of a line segment in $\partial \Omega$. 
  \end{enumerate}
  \end{lemma}
  
  \begin{proof}
Let $\{a_n,b_n\} = \partial \Omega \cap \overline{p_n q_n}$ labelled so that 
\begin{align*}
d_{\Omega}(p_n, q_n) = \log \frac{ \abs{p_n-b_n}\abs{q_n-a_n}}{\abs{p_n-a_n}\abs{q_n-b_n}}.
\end{align*}
We may suppose that $p \neq q$ (otherwise there is nothing to prove). By passing to a subsequence we can suppose that $a_n \rightarrow a$ and $b_n \rightarrow b$. Since $p \neq q$ we see that $a \neq q$ and $b \neq p$. Then 
\begin{align*}
\lim_{n \rightarrow \infty} d_{\Omega}(p_n, q_n) 
&=  \log\left( \abs{p-b}\abs{q-a}\right) + \lim_{n \rightarrow \infty} \log \frac{1}{\abs{p-a_n}\abs{q-b_n}} \\
\end{align*}
Since the limit is finite, we must have that $p \neq a$ and $q \neq b$ which implies that $p$ and $q$ are in the interior of a line segment with end points $a$ and $b$ in $\partial \Omega$.  
\end{proof}

\begin{definition}
Suppose $\Omega$ is a proper convex open set. The \emph{extreme points} of $\Omega$, denoted $\Ext(\Omega)$, are the points in $\partial \Omega$ which are not contained in the interior of a line segment in $\partial \Omega$. 
\end{definition}

\begin{remark} If $\Omega \subset \Pb(\Rb^{d+1})$ is a proper convex open set then there exists an affine chart of $\Pb(\Rb^{d+1})$ which contains $\Omega$ as a bounded convex open set. In this affine chart $\xi \in \Ext(\Omega)$ if and only if $\xi$ is an extreme point of $\overline{\Omega}$ in the usual sense.
\end{remark}

\begin{lemma}\label{lem:orb_ext}
Suppose $\Omega$ is a proper convex open set and $G \leq \Aut(\Omega)$ is a subgroup which acts co-compactly on $\Omega$.  Then $\Ext(\Omega) \subset \overline{G \cdot x}$ for every $x \in \Omega$. Moreover, if $G$ does not have any invariant proper projective subspaces then $\Ext(\Omega) \subset \overline{G \cdot x}$ for every $x \in \overline{\Omega}$
\end{lemma}

\begin{remark} By Theorem~\ref{thm:irred}, if $M = \Gamma \backslash \Omega$ is a compact proper convex real projective manifold and $\Omega$ is irreducible then $\Gamma$ will not have any invariant proper projective subspaces.\end{remark}

\begin{proof}
Suppose that $\xi \in \Ext(\Omega)$ is an extreme point in $\partial \Omega$. Then there exists $p_n \in \Omega$ such that $p_n \rightarrow \xi$. Fix some $x_0 \in \Omega$, then we can find a sequence $\varphi_n \in G$ such that 
\begin{align*}
d_{\Omega}(\varphi_n x_0, p_n) \leq R
\end{align*}
for some $R < \infty$. Now for any $x \in \Omega$ we have 
\begin{align*}
d_{\Omega}(\varphi_n x, p_n) \leq d_{\Omega}(\varphi_n x, \varphi_n x_0) + d_{\Omega}(\varphi_n x_0, p_n) \leq d_{\Omega}(x,x_0)+R.
\end{align*}
And so Lemma~\ref{lem:bd_behavior} implies that $\varphi_n x \rightarrow \xi$. Since $\xi$ was an arbitrary extreme point, this completes the proof in the case in which $x \in \Omega$.

Now suppose that $G$ does not have any proper invariant projective subspaces. Let $\wh{\varphi}_n \in \GL_{d+1}(\Rb)$ be representatives of $\varphi_n$ such that $\norm{\wh{\varphi}_n}=1$. By passing to a subsequence we can suppose that $\wh{\varphi}_n \rightarrow S$ in $\End(\Rb^{d+1})$. Now if $x \in \Pb(\Rb^{d+1}) \setminus [\ker S]$ then $[S](x) = \lim_{n \rightarrow \infty} \varphi_n x$. But we know that $\varphi_n x \rightarrow \xi$ for all $x \in \Omega$. Since $\Omega$ is an open set this implies that $\Imag(S)=\xi$ (viewing $\xi$ as a line in $\Rb^{d+1}$).

Now suppose that $x \in \partial \Omega$. Since $G$ has no proper invariant projective subspaces there exists $\phi \in G$ such that $\phi x \notin [\ker S]$. Then $\varphi_n \phi x \rightarrow [S](x) = \xi$. Since $\xi$ was an arbitrary extreme point, this completes the proof.
\end{proof}

\subsection{The Hilbert metric on a convex cone}

Recall that a convex open cone $\Cc \subset \Rb^{d}$ is called proper if $\Cc$ does not contain any affine lines. If $\Cc \subset \Rb^d$ is a proper convex open cone and we view $\Rb^d$ as an affine chart of $\Pb(\Rb^{d+1})$ then $\Cc$ will be a proper convex open set of $\Pb(\Rb^{d+1})$. In particular, $\Cc$ has a Hilbert metric $d_{\Cc}$ which has all the properties established above. 

\section{Proof of Proposition~\ref{prop:compact}}

Suppose $M_1=\Gamma_1 \backslash \Omega_1$ and $M_2 = \Gamma_2 \backslash \Omega_2$ are compact proper convex real projective manifolds and $f_n : M_1 \rightarrow M_2$ is a sequence of projective maps. By Lemma~\ref{lem:contract} each $f_n$ is distance decreasing with respect to the Hilbert metric, so we may pass to a subsequence so that $f_n \rightarrow f$ in the compact-open topology. 

We claim that $f$ is projective. This follows immediately from the local version of the fundamental theorem of projective geometry, but we will provide a direct proof. First lift each $f_n$ to a projective map $\wt{f}_n : \Omega_1 \rightarrow \Omega_2$. By choosing the lifts correctly, we may assume that $\wt{f}_n \rightarrow \wt{f}$ where $\wt{f}$ is a lift of $f$. Now by Proposition~\ref{prop:proj_map} there exists $T_n \in \Pb(\Lin(\Rb^{d_1+1}, \Rb^{d_2+1}))$ so that $\wt{f}_n(p) = T_n(p)$ for all $p \in \Omega_1$. We can pick a representative $S_n \in \Lin(\Rb^{d_1+1}, \Rb^{d_2+1})$ of $T_n$ so that $\norm{S_n}=1$ and then pass to a subsequence so that $S_n \rightarrow S$ in $\Lin(\Rb^{d_1+1}, \Rb^{d_2+1})$. Then for $p \in \Omega_1 \setminus [\ker S]$ we have that $\wt{f}(p) = [S](p)$. By Lemma~\ref{lem:contract}
\begin{align*}
d_{\Omega_2}(\wt{f}(p), \wt{f}(q)) \leq d_{\Omega_1}(p,q)
\end{align*}
for all $p,q \in \Omega_1$. Since $\wt{f}$ and $[S]$ agree on $\Omega_1 \backslash [\ker S]$ we see that 
\begin{align*}
\{ x \in \Omega_1 : \lim_{y \rightarrow x} [S](x) \text{ does not exist}\} = \emptyset.
\end{align*}
So by Lemma~\ref{lem:cont} we see that $\Omega_1 \cap [\ker S]= \emptyset$ and hence that $\wt{f}$ is a projective map. Since $\wt{f}$ is a lift of $f$ we then see that $f$ is a projective map. 

Finally, since $M_2$ is compact each homotopy class is open in the space of continuous functions and so the set of projective maps from $M_1$ to $M_2$ can contain only finitely many distinct homotopy classes. 

\section{Proof of Theorem~\ref{thm:main_structure}}

Suppose $M_1=\Gamma_1 \backslash \Omega_1$ and $M_2 = \Gamma_2 \backslash \Omega_2$ are compact proper convex real projective manifolds and $f_0 : M_1 \rightarrow M_2$ is a projective map. As in Subsection~\ref{subsec:proj_maps}, we can lift $f_0$ to a map $T_0:\Omega_1 \rightarrow \Omega_2$ and $T_0$ is $\rho$-equivariant for some homomorphism $\rho: \Gamma_1 \rightarrow \Gamma_2$.

Let 
\begin{align*}
\Proj(\Omega_1, \Omega_2) \subset \Pb(\Lin(\Rb^{d_1+1}, \Rb^{d_2+1}))
\end{align*}
be the space of projective maps with $T(\Omega_1) \subset  \Omega_2$. Now $\Gamma_2$ acts on $\Proj(\Omega_1, \Omega_2)$ by $(\varphi \cdot T)(x) = \varphi \cdot (T(x))$. The action is proper and co-compact because the action of $\Gamma_2$ on $\Omega_2$ is proper and co-compact. 

Next let
\begin{align*}
\Proj(\Omega_1, \Omega_2)^{\rho}\subset \Proj(\Omega_1, \Omega_2)
\end{align*}
be the subset of $\rho$-equivariant projective maps. Given any $f \in [f_0]$ we can lift $f$ to a $\rho$-equivariant projective map $T:\Omega_1 \rightarrow \Omega_2$. Conversely any  $\rho$-equivariant projective map $T:\Omega_1 \rightarrow \Omega_2$ descends to a projective map $f: M_1 \rightarrow M_2$ which is homotopic to $f_0$. In particular we can identify $[f_0]$ with
\begin{align*}
G \backslash \Proj(\Omega_1,\Omega_2)^{\rho}
\end{align*}
where $G := \{ \gamma \in \Gamma_2 : \gamma \Proj(\Omega_1,\Omega_2)^{\rho} = \Proj(\Omega_1,\Omega_2)^{\rho}\}$. Since $f_0$ is non-constant, Proposition~\ref{prop:proj_map} implies that $\rho$ is non-trivial. Thus each $T \in \Proj(\Omega_1,\Omega_2)^{\rho}$ is non-constant and has rank at least two. Thus by Lemma~\ref{lem:rank2} the compact-open topology on $[f_0]$ coincides with the quotient topology on $G \backslash \Proj(\Omega_1,\Omega_2)^{\rho}$. 

Now for $\varphi \in \Gamma_2$ let $\rho_{\varphi} : \Gamma_1 \rightarrow \Gamma_2$ be the homomorphism $\rho_{\varphi}(\gamma) = \varphi \rho(\gamma) \varphi^{-1}$. Then 
\begin{align*}
\varphi \Prop(\Omega_1, \Omega_2)^{\rho} = \Prop(\Omega_1, \Omega_2)^{\rho_{\varphi}}
\end{align*}
so $G=C_{\Gamma_2}(\rho(\Gamma_1))$ and $[f_0]$ can actually be identified with 
\begin{align*}
C_{\Gamma_2}(\rho(\Gamma_1)) \backslash \Proj(\Omega_1,\Omega_2)^{\rho}.
\end{align*}

All of these observations reduce Theorem~\ref{thm:main_structure} and Corollary~\ref{cor:center} to:

\begin{proposition}\label{prop:center_2}
Suppose that $\Prop(\Omega_1,\Omega_2)^{\rho}$ contains at least two distinct maps. Then $\Proj(\Omega_1,\Omega_2)^{\rho}$ is a proper convex open set of positive dimension in some projective subspace $P$ of $\Pb(\Lin(\Rb^{d_1+1}, \Rb^{d_2+1}))$. Moreover, $C_{\Gamma_2}(\rho(\Gamma_1))$ acts co-compactly on $\Proj(\Omega_1,\Omega_2)^{\rho}$.
\end{proposition}

The Proposition will follow from a series of Lemmas. 

\begin{lemma} The space
\begin{align*}
[f_0] = C_{\Gamma_2}(\rho(\Gamma_1)) \backslash \Proj(\Omega_1,\Omega_2)^{\rho}
\end{align*}
is compact.
\end{lemma}

\begin{proof}
Given a sequence $f_n \in [f_0]$ we can use Proposition~\ref{prop:compact} to pass to a subsequence so that $f_n \rightarrow f$ where $f:M_1 \rightarrow M_2$ is projective map. Since $M_2$ is compact each homotopy class is open, so for $n$ large $f_n$ is homotopic to $f$. Thus $f \in [f_0]$. Since $f_n$ was an arbitrary sequence in $[f_0]$ we see that $[f_0]$ is compact. 
\end{proof}

It remains to show that $\Proj(\Omega_1,\Omega_2)^{\rho}$ is a proper convex open set of positive dimension in some projective subspace $P$ of $\Pb(\Lin(\Rb^{d_1+1}, \Rb^{d_2+1}))$. To establish this it will be helpful to work in the affine world. 

Let $\Cc_1$ and $\Cc_2$ be proper convex open cones above $\Omega_1$ and $\Omega_2$ respectively. As in Section~\ref{sec:cones} we can lift each $\Gamma_i$ to a discrete group $\Lambda_i \leq \Aut(\Cc_i)$ which acts freely, properly discontinuously, and co-compactly on $\Cc_i$. We can also lift $\rho$ to a homomorphism $\tau:\Lambda_1 \rightarrow \Lambda_2$ and $T_0$ to a linear map $S_0:\Cc_1 \rightarrow \Cc_2$ which is $\tau$-equivariant.

Now let $\Lin(\Cc_1, \Cc_2)^{\tau}$ be the set of $\tau$-equivariant linear maps with $S(\Cc_1) \subset \Cc_2$. Then there is a natural map 
\begin{align*}
\pi: \Lin(\Cc_1, \Cc_2)^{\tau} \rightarrow \Proj(\Omega_1, \Omega_2)^{\rho}.
\end{align*}
Clearly any map $T \in \Proj(\Omega_1, \Omega_2)^{\rho}$ can be lifted to a map $S \in \Lin(\Cc_1, \Cc_2)^{\tau}$ and so $\pi$ is onto. 

\begin{lemma}\label{lem:bd_dist}
If $S_1, S_2 \in \Lin(\Cc_1, \Cc_2)^{\tau}$ then there exists $R \geq 0$ so that 
\begin{align*}
d_{\Cc_2}( S_1(x), S_2(x)) \leq R
\end{align*}
for all $x \in \Cc_1$.
\end{lemma}

\begin{proof}
This follows from the fact that $\Lambda_1$ acts co-compactly on $\Cc_1$ and the fact that both maps are $\tau$-equivariant. 
\end{proof}

\begin{lemma} If there exists two distinct $T_1, T_2 \in \Proj(\Omega_1, \Omega_2)^{\rho}$ then $\Proj(\Omega_1, \Omega_2)^{\rho}$ is a proper convex open set of positive dimension in some projective subspace $P$ of $\Pb(\Lin(\Rb^{d_1+1}, \Rb^{d_2+1}))$.
\end{lemma}

\begin{proof}
With the notation above, let $V \leq \Lin(\Rb^{d_1+1}, \Rb^{d_2+1})$ be the smallest linear subspace which contains $\Lin(\Cc_1, \Cc_2)^{\tau}$. It is enough to show that $\Lin(\Cc_1, \Cc_2)^{\tau}$ is an open proper convex cone in $V$.

Notice that $\Lin(\Cc_1, \Cc_2)^{\tau}$ cannot contain any affine lines because $\Cc_2$ is proper. Moreover, $\Lin(\Cc_1, \Cc_2)^{\tau}$ is closed under scalar multiplication by a positive number. Thus we only need to show that $\Lin(\Cc_1, \Cc_2)^{\tau}$ is an open convex subset of $V$. To establish this it is enough to prove: for any $S_1, S_2 \in \Lin(\Cc_1, \Cc_2)^{\tau}$ there exists $\epsilon >0$ such that 
\begin{align*}
\lambda S_1 + (1-\lambda)S_2 \in \Lin(\Cc_1, \Cc_2)^{\tau}
\end{align*}
for every $\lambda \in (-\epsilon, 1+\epsilon)$. 

So suppose that $S_1, S_2 \in \Lin(\Cc_1, \Cc_2)^{\tau}$. Since $\Cc_2$ is open and convex, for any $p \in \Cc_1$ there exists a maximal $\delta(p) \in (0,\infty]$ such that 
\begin{align*}
\lambda S_1(p) + (1-\lambda)S_2(p) \in \Cc_2
\end{align*}
for $\lambda \in (-\delta(p), 1+\delta(p))$.

We claim that $\delta(p)$ is bounded from below. By Lemma~\ref{lem:bd_dist} there exists $R \geq 0$ such that 
\begin{align*}
d_{\Cc_2}( S_1(p), S_2(p)) \leq R
\end{align*}
for all $p \in \Cc_1$. Now let $p \in \Cc_1$ and $\{a, b\} = \partial \Cc_2 \cap \overline{S_1(p)S_2(p)}$ ordered $a, S_1(p), S_2(p), b$ along the line $\overline{S_1(p)S_2(p)}$ (here it is possible for one of $a$ or $b$ to be $\infty$). Then
\begin{align*}
R \geq 
& d_{\Cc_2}(S_1(p),S_2(p)) =  \log \frac{ \abs{S_1(p)-b}\abs{S_2(p)-a}}{\abs{S_1(p)-a}\abs{S_2(p)-b}} \\
& = \log \frac{ \abs{S_1(p)-b}}{\abs{S_2(p)-b}}+\log \frac{\abs{S_2(p)-a}}{\abs{S_1(p)-a}}.
\end{align*}
We may assume that $S_1(p) \neq S_2(p)$ because otherwise $\delta(p) = \infty$. By possibly relabeling $S_1$ and $S_2$ we can also assume that $\delta(p) \abs{S_1(p)-S_2(p)} = \abs{S_1(p)-a}$. Then 
\begin{align*}
R \geq d_{\Cc_2}(S_1(p),S_2(p)) \geq  \log \frac{\abs{S_2(p)-a}}{\abs{S_1(p)-a}} \geq \log \frac{\abs{S_2(p)-S_1(p)}}{\abs{S_1(p)-a}} \geq \log \frac{1}{\delta(p)}.
\end{align*}
Hence $\delta(p) \geq e^{-R}$. 

Thus $S = \lambda S_1 + (1-\lambda)S_2$ is in $ \Lin(\Cc_1, \Cc_2)^{\tau}$ for every $\lambda \in (-e^{-R}, 1+e^{-R})$. 
\end{proof}

\section{Proof of Proposition~\ref{prop:factor} and Theorem~\ref{thm:prod_or_aut}}

We begin by proving Proposition~\ref{prop:factor} whose statement we recall:

\begin{proposition}\label{prop:factor_2}
Suppose that $f_0: M_1 \rightarrow M_2$ is a non-constant projective map between two compact proper convex real projective manifolds. Then there exists a compact proper convex real projective manifold $N$ and a surjective projective map $p:M_1 \rightarrow N$ with the following property: if $f \in [f_0]$ then there exists a locally injective projective map $\overline{f}:N \rightarrow M_2$ with $f = \overline{f} \circ p$.
\end{proposition}

We will need one lemma:

\begin{lemma}
Suppose $\Omega \subset \Pb(\Rb^d)$ is a proper convex open set and $G \leq \Aut(\Omega)$ is a subgroup. If $G$ acts freely on $\Omega$ then $G$ is torsion free. 
\end{lemma}

\begin{proof}
Suppose that $g^n =1$ for some $g\in G$. For $p \in \Omega$ let $C_p$ be the convex hull of the points $\{ g^z p : z \in \Zb\}$. Then $C_p \subset \Omega$ is compact, convex, and $g$-invariant. So by the Brouwer fixed-point theorem $g$ has a fixed point in $C_p$. Since $G$ acts freely on $\Omega$ we see that $g=1$.
\end{proof}

\begin{proof}[Proof of Proposition~\ref{prop:factor_2}]
Assume $M_1 = \Gamma_1 \backslash \Omega_1$ and $M_2 = \Gamma_2 \backslash \Omega_2$. As in Proposition~\ref{prop:proj_map} we can lift $f_0$ to a $\rho$-equivariant map $T_0:\Omega_1 \rightarrow \Omega_2$  for some homomorphism $\rho: \Gamma_1 \rightarrow \Gamma_2$. 

Let $\Cc_1$ and $\Cc_2$ be proper convex open cones above $\Omega_1$ and $\Omega_2$ respectively. As in Section~\ref{sec:cones} we can lift each $\Gamma_i$ to a discrete group $\Lambda_i \leq \Aut(\Cc_i)$ which acts freely, properly discontinuously, and co-compactly on $\Cc_i$. We can also lift $\rho$ to a homomorphism $\tau:\Lambda_1 \rightarrow \Lambda_2$ and $T_0$ to a linear map $S_0:\Cc_1 \rightarrow \Cc_2$ which is $\tau$-equivariant.

By Proposition~\ref{prop:proj_map}, $U:=\ker S_0$ is a $\Lambda_1$-invariant subspace and by Thereom~\ref{thm:irred} there exists a $\Lambda_1$-invariant decomposition $\Rb^{d_1+1} = U \oplus W$ and proper convex cones $\Cc_U \subset U$ and $\Cc_W \subset W$ such that $\Cc_1 = \Cc_U + \Cc_W$. Let $p_U: \Rb^{d_1+1} \rightarrow U$ and $p_W:\Rb^{d_1+1} \rightarrow W$ be the natural projections and $\pi_U: \Lambda_1 \rightarrow \GL(U)$ and $\pi_W: \Lambda_1 \rightarrow \GL(W)$ be the natural restrictions. 
 
 Now let $V:=S_0(\Rb^{d_1+1})$ then $S_0$ descends to a linear isomorphism $\overline{S}_0: W \rightarrow V$ which maps $\Cc_W$ into $\Cc_2 \cap V$. Moreover
 \begin{align*}
 \tau(\phi)\overline{S}_0(w) =  \tau(\phi)S_0(u,w) = S_0( \phi \cdot (u,w)) = \overline{S}_0(\pi_W(\phi) w)
 \end{align*}
 for $\phi \in \Lambda_1$ and $(u,w) \in \Cc_U + \Cc_W$. In particular, if $\pi_W(\phi)=1$ then $\tau(\phi)\overline{S}_0(w) = \overline{S}_0(w)$ for every $w \in \Cc_W$. Since $\Lambda_2$ acts freely on $\Cc_2$ we see that if $\pi_W(\phi)=1$ then $\tau(\phi)=1$. Thus there exists a homomorphism $\overline{\tau}: \pi_W(\Lambda_1) \rightarrow \Lambda_2$ such that $\tau = \overline{\tau} \circ \pi_W$. The above formula then implies that
  \begin{align*}
\overline{ \tau}(\phi) \circ \overline{S}_0 =  \overline{S}_0 \circ \phi
 \end{align*}
 and the injectivity of $\overline{S}_0$ implies that $\overline{\tau}$ is injective. 
 
Now let 
\begin{align*}
G := \{ \varphi|_V : \varphi \in \Lambda_2, \varphi(V) = V\} \leq \Aut(\Cc_2 \cap V).
\end{align*}
Since the action of $G$ on $\Cc_2 \cap V$ is properly discontinuous and free we see that $G$ is discrete and torsion-free. Then 
\begin{align*}
\pi_W(\Lambda_1) \leq \overline{S}_0^{-1}\Big( (\overline{\tau} \circ \pi_W)(\Lambda_1)\Big) \overline{S}_0 \leq \overline{S}_0^{-1} G \overline{S}_0
\end{align*} 
is also discrete and torsion-free. Now let $\Omega_W$ be the image of $\Cc_W$ in $\Pb(W)$ and $\Gamma_W$ be the image of $\pi_W(\Lambda_1)$ in $\PGL(W)$. Since $\{ e^z \Id : z \in \Zb \} \leq \pi_W(\Lambda_1)$ the group $\Gamma_W$ is discrete and torsion-free in $\PGL(W)$. Then $N := \Gamma_W \backslash \Omega_W$ is a proper convex real projective manifold. Moreover the map $p_W : \Cc_1 \rightarrow \Cc_W$ descends to a projective map $p:M_1 \rightarrow N$ and so $N$ is compact. 

Now suppose that $f \in [f_0]$ then we can lift $f$ to a linear map $S: \Cc_1 \rightarrow \Cc_2$ which is $\tau$-equivariant. We first claim that $\ker S = \ker S_0$. Since $\Lambda_1$ acts co-compactly on $\Cc_1$ and $S,S_0$ are $\tau$-equivariant there exists $R \geq 0$ such that 
\begin{align*}
d_{\Cc_2}(S(p), S_0(p)) \leq R
\end{align*}
for all $p \in \Cc_1$. Then if $v \in \ker S_0$ we have 
\begin{align*}
d_{\Cc_2}(S(p)+tS(v), S_0(p)) \leq R
\end{align*}
whenever $p+tv \in \Cc_1$. Which implies, by the properness  of the metric $d_{\Cc_2}$, that $S(v)=0$. Thus $\ker S_0 \subset \ker S$. Switching the roles of $S_0$ and $S$ in the above argument shows that $\ker S \subset \ker S_0$. So $\ker S = \ker S_0$ and there exists an injective linear map $\overline{S} : W \rightarrow \Rb^{d_2+1}$ such that $S = \overline{S} \circ p_W$. Then $\overline{S}$ is $\overline{\tau}$-equivariant and the map $\overline{S}: \Cc_W \rightarrow \Cc_2$ descends to a map $\overline{f}:N \rightarrow M_2$ such that $f=\overline{f} \circ p$.
 \end{proof}

We can now prove Theorem~\ref{thm:prod_or_aut} whose statement we recall:

\begin{proposition}
With the notation in Proposition~\ref{prop:factor}, $[f_0] \neq \{f_0\}$ if and only if either 
\begin{enumerate}
\item $\Aut(N)$ is infinite 
\item there exists a compact proper convex real projective manifold $L$ and a continuous locally injective map $F:N \times L \rightarrow M_2$ such that 
\begin{enumerate}
\item for any fixed $\ell \in L$ the map $n \in N \rightarrow F(n,\ell)$ is projective,
\item for any fixed $n \in N$ the map $\ell \in L \rightarrow F(n,\ell)$ is projective, and
\item there exists $\ell_0 \in L$ such that $f_0(m) = F(p(m),\ell_0)$ for all $m \in M_1$.
\end{enumerate}
\end{enumerate}
\end{proposition}

\begin{proof} We will freely use the notation from the proof of Proposition~\ref{prop:factor_2}. Let $\Lambda_W:=\pi_W(\Lambda_1)$. \newline 

\noindent \textbf{Case 1:} $\Lambda_W \leq \GL(W)$ does not act irreducibly on $W$. Then by Theorem~\ref{thm:irred} there exists a $\Lambda_W$-invariant decomposition $W = W_1 \oplus W_2$ and proper convex open cones $\Cc_{W,1} \subset W_1$ and $\Cc_{W,2} \subset W_2$ such that $\Cc_{W} = \Cc_{W,1} + \Cc_{W,2}$. Then, since this decomposition is $\Lambda_W$-invariant, we see that the group
\begin{align*}
H := \left\{ \begin{pmatrix} a \Id_{W_1} & 0 \\ 0 & b\Id_{W_2} \end{pmatrix} : a,b >0\right\}
\end{align*}
is a subgroup of $\Aut(\Cc_W)$ and centralizes $\Lambda_W$. Then the action of $H$ on $\Cc_W$ descends to a projective action on $N$ and hence $\Aut(N)$ is infinite. \newline
 
\noindent \textbf{Case 2:}  $\Lambda_W \leq \GL(W)$ acts irreducibly on $W$. Suppose $f_0 = \overline{f}_0 \circ p$. Let $L := [\overline{f}_0]$ and $F:N \times L \rightarrow M_2$ be the map given by $F(n,\ell) = \ell(n)$. By Theorem~\ref{thm:main_structure}
\begin{enumerate}
\item  $L$ is a compact proper convex real projective manifold, 
\item $F$ is continuous,
\item for any fixed $\ell \in L$ the map $n \in N \rightarrow F(n,\ell)$ is projective,
\item for any fixed $n \in N$ the map $\ell \in L \rightarrow F(n,\ell)$ is projective, and
\item $f(m) = F(p(m), \overline{f}_0)$ for any $m \in M_1$.
\end{enumerate}

It remains to show that $F$ is locally injective. To see this lift $F$ to the map $\wt{F}: \Omega_1 \times \Proj(\Omega_1, \Omega_2)^{\rho} \rightarrow \Omega_2$ given by $\wt{F}(p,T) = T(p)$. Then it is enough to show that $\wt{F}$ is injective. So suppose that $T_1(p_1) = T_2(p_2)$ for some $p_1, p_2 \in \Omega_1$ and $T_1,T_2 \in  \Proj(\Omega_1, \Omega_2)^{\rho}$. 

Given a subset $A \subset \Pb(\Rb^{d+1})$ let $\ip{A}$ be the smallest projective subspace containing $A$. Notice that $\ip{T_i(\Omega_1)} = \ip{T_i(\Ext(\Omega_1))}$. Moreover, by Lemma~\ref{lem:orb_ext} and the fact that 
\begin{align*}
T_1(\phi \cdot p_1) = \rho(\phi)T_1(p_1) = \rho(\phi)T_2(p_2) = T_2( \phi \cdot p_2) \subset \ip{ T_2(\Omega_1)},
\end{align*}
we see that 
\begin{align*}
T_1(\Ext(\Omega_1)) \subset \ip{ T_2(\Omega_1)}.
\end{align*}
Similarly, $T_2(\Ext(\Omega_1)) \subset \ip{ T_2(\Omega_1)}$. Thus $T_1$ and $T_2$ have the same image. Let $P = \ip{ T(\Omega_1)} = \ip{T(\Omega_2)}$. Now by the proof of Proposition~\ref{prop:factor_2} the maps $T_1$ and $T_2$ yield $\overline{\tau}$-equivariant maps $\overline{T}_1 : \Omega_W \rightarrow \Omega_2 \cap P$ and $\overline{T}_2 : \Omega_W \rightarrow \Omega_2 \cap P$.  Notice that $\overline{T}_1, \overline{T}_2 : \Pb(W) \rightarrow P$ are isomorphisms. Then $\Phi = \overline{T}_1^{-1} \circ \overline{T}_2 : \Omega_W \rightarrow \Omega_W$ is an $\Gamma_W$-invariant map. We claim that $\Phi = [\Id_W]$. To see this first observe that there exists $R \geq 0$ such that 
\begin{align*}
d_{\Omega_W}(\Phi(p), p) \leq R
\end{align*}
for all $p \in \Omega_W$. This follows from the fact that $\Gamma_W$ acts co-compactly on $\Omega_W$. Then by Lemma~\ref{lem:bd_behavior} we see that $\Phi(e) = e$ for every extreme point of $\Omega_W$. Now fix a basis $v_1, \dots, v_k$ of $W$ such that $[v_1],  \dots, [v_k]$ are all extreme points of $\Omega_W$. Then with respect to this basis $\Phi$ is given by a  matrix of the form 
\begin{align*}
\Phi = \begin{bmatrix} \alpha_1 & &  \\ & \ddots & \\ & & \alpha_k \end{bmatrix}.
\end{align*}
Since $\Phi$ commutes with $\Gamma_W$, $\Gamma_W$ preserves the projective subspace
\begin{align*}
E_{\lambda} = [\operatorname{Span}(v_i : \alpha_i = \lambda)]
\end{align*}
for any $\lambda \in \Rb$. But since $\Gamma_W$ acts irreducibly on $\Pb(W)$ we must have that 
\begin{align*}
\alpha_1 = \dots = \alpha_k
\end{align*}
and hence that $\Phi = [\Id_W]$. But this implies that $\overline{T}_1 = \overline{T}_2$ and hence that $T_1 = T_2$. 
\end{proof}

\section{Proof of Theorem~\ref{thm:main}}

Using Proposition~\ref{prop:compact} and Proposition~\ref{prop:factor} it is enough to show:

\begin{lemma}
Suppose $f_0: M_1 \rightarrow M_2$ is a locally injective projective map between two compact proper convex real projective manifolds. If $M_2$ is strictly convex then $[f_0]=\{f_0\}$.  
\end{lemma}

\begin{proof}
Assume $M_1 = \Gamma_1 \backslash \Omega_1$ and $M_2 = \Gamma_2 \backslash \Omega_2$. As in Subsection~\ref{subsec:proj_maps}, we can lift $f_0$ to a map $T_0:\Omega_1 \rightarrow \Omega_2$ and $T_0$ is $\rho$-equivariant for some homomorphism $\rho: \Gamma_1 \rightarrow \Gamma_2$. Since $f_0$ is locally injective, $T_0$ has full rank and hence is injective. Then since $T_0$ is injective, so is $\rho$. 

We first claim that $T_0(\partial \Omega_1) \subset \partial \Omega_2$. Suppose not, then there exists some $x \in  \partial \Omega_1$ so that $T_0(x) \in \Omega_2$. Pick $p_n \in \Omega_1$ so that $p_n \rightarrow x$. Then there exists $R_0 \geq 0$ so that
\begin{align*}
d_{\Omega_2}(T_0(p_n), T_0(x)) \leq R_0
\end{align*}
for all $n \in \Nb$.  There also exists $\varphi_n \in \Gamma_1$ and $R_1 \geq 0$ so that 
\begin{align*}
d_{\Omega_1}(p_n, \varphi_n p_1) \leq R_1
\end{align*}
for all $n \in \Nb$. Notice that $\varphi_n \rightarrow \infty$ in $\PGL_{d_1+1}(\Rb)$ because $\Aut(\Omega_1)$ is a closed subgroup and $x \in \partial \Omega$. Then
\begin{align*}
d_{\Omega_2}(\rho(\varphi_n) T_0(p_1), T_0(x)) \leq d_{\Omega_2}(T_0(\varphi_n p_1), T_0(p_n)) + d_{\Omega_2}(T_0(p_n), T_0(x)) \leq R_0+R_1.
\end{align*}
But $\Gamma_2$ acts properly discontinuously on $\Omega_2$ and so the set $\{ \rho(\varphi_n)  : n \in \Nb\}$ is finite. Since $\rho$ is injective, the set $\{ \varphi_n : n \in \Nb\}$ is also finite which contradicts the fact that $\varphi_n \rightarrow \infty$ in $\PGL_{d+1}(\Rb)$. Thus $T_0(\partial \Omega_1) \subset \partial \Omega_2$.

Given $f_1 \in [f_0]$ we can lift $f_1$ to a $\rho$-equivariant projective map $T_1:\Omega_1 \rightarrow \Omega_2$. Since both $T_0$ and $T_1$ are $\rho$-equivariant and $\Gamma_1$ acts co-compactly on $\Omega_1$ we see that there exists $R_2 \geq 0$ so that
\begin{align*}
d_{\Omega_2}(T_1(x), T_2(x)) \leq R_2
\end{align*}
for all $x \in \Omega_1$. 

Next we claim that $T_0|_{\partial \Omega_1} = T_1|_{\partial \Omega_1}$. Fix a point $x \in \partial \Omega$ and a sequence $p_n \in \Omega$ so that $p_n \rightarrow x$. Suppose that $y=T(x)$. Since $y \in \partial \Omega_2$ and $\Omega_2$ is strictly convex, $y$ is an extreme point of $\Omega_2$. Moreover, 
\begin{align*}
d_{\Omega_2}(T_0(p_n), T_1(p_n)) \leq R_2
\end{align*}
for all $n \in \Nb$. So by Lemma~\ref{lem:bd_behavior} we see that $T_1(x) = \lim_{n \rightarrow \infty} T_1(p_n)= y$. 

Finally we clam that $T_0=T_1$. For a point $p \in \Omega_1$ let $\ell_1$ and $\ell_2$ be two distinct projective lines through $p$. Then
\begin{align*}
T_i(p) = T_i(\ell_1) \cap T_i(\ell_2).
\end{align*}
However the projective line $T_i(\ell_j)$ is completely determined by $T_i(\ell_j \cap \partial \Omega_1)$ and $T_0|_{\partial \Omega_1} = T_1|_{\partial \Omega_1}$. Hence we see that $T_0 = T_1$. 
\end{proof}

\bibliographystyle{alpha}
\bibliography{hilbert}

\end{document}